%% file: LASSO_corr06_Short.tex
\documentclass[journal,10pt,onecolumn,draftclsnofoot,]{IEEEtran}

\IEEEoverridecommandlockouts
\usepackage{cite}
\usepackage{amsmath,amssymb,amsfonts}
\usepackage{algorithmic}
\usepackage{graphicx}
\usepackage{textcomp}
\usepackage{xcolor}
\usepackage{amsthm}
\def\BibTeX{{\rm B\kern-.05em{\sc i\kern-.025em b}\kern-.08em
    T\kern-.1667em\lower.7ex\hbox{E}\kern-.125emX}}
    
\usepackage{systeme}
\usepackage{color}
\usepackage{pgfplots}
\usepackage{bbm}
    \DeclareMathOperator{\tr}{tr}
  \newcommand{\figref}[1]{Fig.~\protect\ref{#1}}
\newcommand{\bPhi}{\boldsymbol{\Phi}}

\newcommand{\bSig}{\boldsymbol{\Sigma}}

\newcommand{\bg}{{\bf g}}

      \newcommand{\bx}{{\bf x}}

          \newcommand{\bw}{{\bf w}}

           \newcommand{\bI}{{\bf I}}
    
       \newcommand{\bC}{{\bf C}}

                            \newcommand{\pto}{\overset{P}\longrightarrow }

\include{macros}

\newtheorem{theorem}{Theorem}

\newtheorem{remark}{Remark}

\newtheorem{prop}{Proposition}
\begin{document}
\makeatother
\def\x{{\mathbf x}}
\def\L{{\cal L}}

\title{
Precise Error Analysis of the LASSO under Correlated Designs
}
\author{\IEEEauthorblockN{Ayed M. Alrashdi, Houssem Sifaou, Abla Kammoun, Mohamed-Slim Alouini and Tareq Y. Al-Naffouri}
\IEEEauthorblockA{
{Electrical Engineering Department, University of Hail, Hail, Saudi Arabia} \\
{CEMSE Division, King Abdullah University of Science and Technology (KAUST),} 
Thuwal, Saudi Arabia\\
Emails: \{ayed.alrashdi, houssem.sifaou, abla.kammoun, slim.alouini, tareq.alnaffouri\}@kaust.edu.sa
}
}

\maketitle
\begin{abstract}
In this paper, we consider the problem of recovering a sparse signal from noisy linear measurements using the so called LASSO formulation. We assume a correlated Gaussian design matrix with additive Gaussian noise.
We precisely analyze the high dimensional asymptotic performance of the LASSO under correlated design matrices using the Convex Gaussian Min-max Theorem (CGMT). We define appropriate performance measures such as the mean-square error (MSE), probability of support recovery, element error rate (EER) and cosine similarity.
Numerical simulations are presented to validate the derived theoretical results.
\end{abstract}

\section{Introduction}
\label{sec:intro}

The LASSO is one of the most celebrated methods in statistics and signal processing \cite{tibshirani1996regression}. Given a noisy linear measurments $
\yv = \Am \xv_0 + \zv,
$
it recovers the unknown $k$-sparse signal $\xv_0 \in \mathbb{R}^n$ by solving the following optimization problem:
\begin{subequations}\label{eq:corr-LASSO}
\begin{align}
&\widehat{\xv}: = {\rm{arg}} \min_{\xv }  \| \yv - \Am \xv \|_2^2+ \lambda \| \xv \|_1,
\end{align}
\end{subequations}
where $\Am \in \mathbb{R}^{m \times n}$ is the design (measurement) matrix, $\zv \in \mathbb{R}^m$ is the noise vector that has iid entries $\mathcal{N}(0,\sigma^2)$, $\lambda > 0$ is the regularization parameter, $\| \cdot \|_2$ denotes the $\ell_2$-norm of a vector, and $\| \cdot\|_1$ represents its $\ell_1$-norm.

The asymptotic performance of the LASSO has been recently extensively studied in many works. One approach that is based on the Approximate Massage Passing (AMP) framework was used in series of papers to sudy the LASSO under the assumption of iid design matrix \cite{bayati2011lasso,bayati2011dynamics,donoho2009message}. Another approach used the Convex Gaussian min-max Theorem (CGMT) to derive sharp performance garantees of the LASSO for iid design matrices \cite{thrampoulidis2018precise,atitallah2017box,stojnic2013framework,stojnic2010recovery,thrampoulidis2015lasso}. In addition, in \cite{alrashdi2017precise,alrashdi2019precise} the LASSO was analyzed for imperfect designs. In many practical situations, the design matrix has correlated entries \cite{shin2006capacity} so it is important to take into account correlations in the analysis. Very recently, \cite{alrashdi2020box} used the CGMT farmework to analyze the Box-Least Squares decoder under the presence of correlations. To the best of our knowledge, the precise error analysis of the LASSO under correlated designs has not been explicitly derived in this contex before.

To close this gap, this paper derives precise aysmptotic error analysis of the LASSO with correlated Gaussian design matrix. In particular, we provide asymptotic expressions of the mean squared error (MSE) of the LASSO. In addition, we study other interseting performance measures such as the probability of support recovery, the element error rate (EER)  and the cosine similarity.

\section{ Problem Formulation}\label{a}
\subsection{System Model}\label{avv}
We consider a noisy linear measurements system $\yv = \Am \xv_0 +\zv \in \mathbb{R}^m$.
The unknown signal vector $\xv_0 \in \mathbb{R}^n$ is assumed to be $k$-sparse, i.e., only $k$ of its entries are sampled iid from a distribution $p_{X_0}$  and the remaining entries are zeros. The noise vector $\zv \in \mathbb{R}^m$ is assumed to have iid entries $\mathcal{N}(0,\sigma^2)$.
In this work, we consider a correlated Gaussian design matrix which can be modeled as \cite{adhikary2013joint,mueller2016linear}
$$
\Am = \bSig^{\frac{1}{2}} \Hm,
$$
 where $\bSig \in \mathbb{R}^{m \times m}$ is known Hermitian nonnegative left correlation matrix, satisfying $\frac{1}{m} \text{tr}(\bSig) = \mathcal{O}(1)$, while
 $\Hm \in \mathbb{R}^{m \times n}$ is a Gaussian matrix with iid entries $\mathcal{N}(0,\frac{1}{n})$. 
The analysis is performed when the system dimensions grow simultaneously to infinity ($m,n, k \to \infty$) at fixed rates: $\frac{m}{n} \to \delta \in (0,\infty)$ and $\frac{k}{n} \to \kappa\in (0,1)$. The signal-to-noise ratio $(\rm SNR)$ is assumed to be constant and given as ${\rm{SNR}}= \frac{\kappa}{\sigma^2}$.
 
\subsection{ Performance Metrics}\label{a}
We consider the following performance metrics:\\
\textbf{Mean squared error}: The recovery {mean squared error} (MSE) measures the deviation of $\widehat{\xv}$ from the true signal $\xv_0$. 
Formally, it is defined as
\begin{equation}
{\rm{MSE}} := \frac{1}{n}\| \widehat{\xv}  - \xv_0\|_2^2.
\end{equation}
\textbf{Support Recovery}: In the problem of sparse recovery, a natural measure of performance that is used in many applications 
 is support recovery, which is defined as identifying whether an entry of $\xv_0$ is on the support (i.e., non-zero), or it is off the support (i.e., zero). The decision is based on the LASSO solution $\widehat{\xv}$: we say the $i^{th}$ entry of $\widehat{\xv}$ is on the support if $| \widehat{x}_{i}| \geq \xi$, where $\xi > 0$ is a user-defined hard threshold on the entries of $\widehat{\xv}.$ Formally, let
\begin{subequations}\label{supp}
\begin{align}
\Phi_{\xi,\text{on}}(\widehat{\xv}) := \frac{1}{k} \sum_{i \in S(\xv_0)} \mathbbm{1}_{\{| \widehat{x}_{i}| \geq \xi \}},\\
\Phi_{\xi,\text{off}}(\widehat{\xv}) := \frac{1}{n-k} \sum_{i \notin S(\xv_0)} \mathbbm{1}_{\{| \widehat{x}_{i}| \leq \xi \}},
\end{align}
\end{subequations}
\noindent
where $\mathbbm{1}_{\{\mathcal{A} \}}$ is the indicator function of a set $\mathcal{A}$, and $S(\xv_0)$ is the support of $\xv_0$, i.e., the set of the non-zero entries of $\xv_0$. In Theorem \ref{lasso_on/off}, we precisely predict the \textit{per-entry} rate of successful on-support and off-support recovery.\\
\textbf{Element Error Rate:} We also consider the ad-hoc performance metric that we call (per) Element-Error Rate (EER) that is the opposite of the probability of successful recovery. After hard thresholding the entries of $\widehat{\xv}$ by $\xi$ as before, we define the EER as follows
$$ {\rm{EER}}_\xi := \frac{1}{k} \sum_{i \in S(\xv_0)} \mathbbm{1}_{\{| \widehat{x}_{i}| < \xi \}} + \frac{1}{n-k} \sum_{i \notin S(\xv_0)} \mathbbm{1}_{\{| \widehat{x}_{i}| > \xi \}}.$$
As we can see, this metric can be linked to the support recovery metrics defined before as follows:
\begin{equation}\label{EER_def}
 {\rm{EER}}_\xi = 2 -\Phi_{\xi,\text{on}}(\widehat{\xv}) -\Phi_{\xi,\text{off}}(\widehat{\xv}).
\end{equation}
\textbf{Cosine Similarity:} We define another metric that is widely used in machine learning which is the cosine similarity between $\widehat{\xv}$ and $\xv_0$.\footnote{This performance measure can also be seen as the \textit{correlation} between the estimator $\widehat{\xv}$ and $\xv_0$.} It is defined as
$$
{\cos}(\angle (\widehat{\xv},\xv_0)) :=\frac{\widehat{\xv}^T \xv_0}{\| \widehat{\xv}\|_2 \| \xv_0 \|_2} \in [-1,+1].
$$
Obviously, we seek estimates that maximize similarity (correlation).

The rest of the paper is organized as follows. In Section \ref{sec:main}, we present and discuss the main results of the paper. Numerical results are included in Section \ref{sec:simu}, while a proof outline is given in Section \ref{sec:proof}. Finally, the paper is concluded in Section \ref{sec:conclusion}.
\section{Main Results}\label{sec:main}
In this section, we summarize the asymptotic analysis of the LASSO in \eqref{eq:corr-LASSO} in terms of its MSE, probability of support recovery, EER and cosine similarity. 
We use the standard notation ${\rm{plim}}_{n \to \infty} X_n = X$ to denote that a sequence of random variables $X_n$ converges in probability towards a constant $X$. Define the spectral decomposition of $\bSig$
as $\bSig = \Um \Gammam \Um^T$. Finally, let $Q(\cdot)$ denote the Gaussian $Q$-function associated with the standard normal probability density function (pdf) $\varphi(x) =\frac{1}{\sqrt{2 \pi}} e^{\frac{-x^2}{2}}$ .
\begin{theorem}[MSE of the LASSO]\label{LASSO_mse} \normalfont
Let $\rm MSE$ denote the mean squared error of the LASSO in \eqref{eq:corr-LASSO} for some fixed but unknown $k$-sparse signal $\xv_0$, then in the limit of $m, n, k \to \infty, m/n\to \delta$, and $k/n \to \kappa$, it holds
\begin{equation}\label{mse_exp}
\underset{n \to \infty}{{\rm{plim}}} \ \rm{MSE}= \alpha_\star,
\end{equation}
where $\alpha_\star$ is the unique solution to the following:
\begin{align*}\label{cost_th1}
\min_{\alpha>0} \max_{\beta>0} \sup_{\chi>0}& \ \ \mathcal{D}(\alpha,\beta,\chi):=\frac{1}{n}\sum_{j=1}^{m} \frac{\gamma_j \alpha + \sigma^2}{1 - \gamma_j  \mu(\alpha,\beta) }  
- \left(\frac{\beta^2}{4} \mu(\alpha,\beta) +\frac{\chi}{2} +\frac{\alpha \beta^2}{2 \chi} \right) 
+\frac{\chi}{\alpha} \mathbb{E}_{\underset{Z \sim \mathcal{N}(0,1)}{X_0 \sim p_{X_0}} } \biggr[e \biggr(X_0 + \frac{\alpha \beta }{ \chi} Z ; \frac{ \lambda \alpha}{\chi}\biggl)    \biggl] ,
\end{align*}
\begin{equation}\label{soft_cost}
e(a ; b) =
\begin{cases} 
      b a - \frac{1}{2} b^2  & ,\text{if}  \ a > b \\      
					
			\frac{1}{2} a^2 & ,\text{if} \  |a| \leq b  \\
			
      -ba - \frac{1}{2} b^2  & ,\text{if} \   a < -b,
\end{cases}
\end{equation}
$\gamma_j$ is the $j$-th eigenvalue of the matrix $\bSig$, and $\mu(\alpha,\beta)$ satisfies:
\begin{align*}
\frac{1}{n} \sum_{j=1}^{m} \frac{\alpha + \frac{\sigma^2}{\gamma_j}}{\left(\frac{1}{\gamma_j}- \mu(\alpha,\beta)\right)^2} -\frac{\beta^2}{4} =0.
\end{align*}
\end{theorem}
\begin{proof}
A proof outline of this theorem is given in Section \ref{sec:proof}.
\end{proof}
\begin{remark}\normalfont
The optimal solutions $\alpha_\star, \beta_\star, \chi_\star$ can be computed numerically by writing the first order optimality conditions, i.e., by solving $\nabla_{(\alpha,\beta,\chi)} \mathcal{D}(\alpha,\beta,\chi) =0.$
\end{remark}

\begin{remark}\normalfont
Theorem \ref{LASSO_mse} allows us to optimally tune the involved parameter such as the regularizer $\lambda$ or the number of normalized measurements $\delta$, etc.. See Fig.1 for an illustration. Note that the MSE expression in \eqref{mse_exp} requires the knowledge of the noise variance $\sigma^2$. However, even if the noise variance is
unknown, we can use some algorithm to estimatie the SNR such as in \cite{suliman2017snr}.
\end{remark}

The following Theorem precisely characterizes the support recovery metrics introduced in (\ref{supp}).
\begin{theorem}[Probability of support recovery]\label{lasso_on/off}
Under the same settings of Theorem \ref{LASSO_mse} and for any fixed $\xi>0$,
and in the limit of
$m, n,k \to \infty, m/n \to \delta$, and $k/n \to \kappa$, it holds that:
\begin{equation*}
\underset{n\to\infty}{{\rm{plim}}}\ \Phi_{\xi,\rm{on}}(\widehat{\xv}) = \mathbb{P} \biggl[\biggl | \eta \biggr(X_0 + \frac{\alpha_\star \beta_\star}{\chi_\star} Z ; \frac{ \lambda \alpha_\star }{\chi_\star}\biggl)  \biggr |   \geq \xi \biggr], 
\end{equation*}
and 
\begin{equation*}
\underset{n\to\infty}{{\rm{plim}}} \ \Phi_{\xi,\rm{off}}(\widehat{\xv}) = \mathbb{P} \biggl[ \biggl| \eta \biggr( \frac{\alpha_\star \beta_\star}{\chi_\star} Z ; \frac{ \lambda \alpha_\star }{\chi_\star} \biggl)  \biggr|   \leq \xi \biggr],
\end{equation*}
where $\eta(a;b)$ is the soft-thersholding function defined as
\begin{equation}\label{shoft_cost}
\eta(a ; b) =
\begin{cases} 
      a-b  & ,\text{if}  \ a > b \\      
					
			0 & ,\text{if} \  |a| \leq b  \\
			
      a+b  & ,\text{if} \   a < -b.
\end{cases}
\end{equation}
\end{theorem}
\begin{proof}
An overview of the proof is given in Section \ref{sec:proof}.
\end{proof}
\begin{remark} \normalfont
 It should be clear that these probabilities are taken over the randomness of $X_0$ and $Z$.
 \end{remark}

The following proposition derives a precise asymptotic characterization of the EER.
\begin{prop}[Element Error Rate]\label{LASSO_EER} \normalfont
Under the same settings of Theorem \ref{lasso_on/off} and for any fixed $\xi>0$, it holds that:
\begin{align}
\underset{n\to\infty}{{\rm{plim}}} \ {\rm{EER}}_\xi & = \mathbb{P} \left[\left| \eta \biggr(X_0 + \frac{\alpha_\star \beta_\star}{\chi_\star} Z ; \frac{ \lambda \alpha_\star }{\chi_\star}\biggl) \right|   < \xi \right] 
+ \mathbb{P} \left[ \left| \eta \biggr( \frac{\alpha_\star \beta_\star}{\chi_\star} Z ; \frac{ \lambda \alpha_\star }{\chi_\star}\biggl)  \right|   > \xi \right].
\end{align}
\end{prop}
\begin{proof}
The proof follows from Theorem \ref{lasso_on/off} and the definition of the EER as given by \eqref{EER_def}.
\end{proof}

In all of the above metrics, we care about the magnitude of the LASSO estimate. However, in many applications, the orientation of the solution matters as well. This is the objective of the next proposition that characterizes the cosine similarity of the LASSO. 
\begin{prop}[Cosine Similarity]\label{LASSO_similar} \normalfont
Under the same settings of Theorem \ref{LASSO_mse}, it holds that:
\begin{equation*}
\underset{n\to\infty}{{\rm{plim}}} \ \cos(\angle (\widehat{\xv},\xv_0))= \frac{\mathbb{E}_{{X_0 },{Z} } \biggr[\eta \biggr(X_0 + \frac{\alpha_\star \beta_\star}{\chi_\star} Z ; \frac{ \lambda \alpha_\star }{\chi_\star}\biggl)  X_0  \biggl]}{\sqrt{\kappa \mathbb{E}_{{X_0} ,{Z}} \biggr[\eta^2 \biggr(X_0 + \frac{\alpha_\star \beta_\star}{\chi_\star} Z ; \frac{ \lambda \alpha_\star }{\chi_\star}\biggl)  \biggl]}}.
\end{equation*}
\end{prop}
\begin{proof}
The proof is based on the CGMT to derive asymptotic predictions of the numerator and the denominator of the cosine similarity expression separetely and then use the Continuous Mapping Theorem to arrive at this proposition. Details are omitted for space limitations.
\end{proof}
\section{Numerical Results}\label{sec:simu}
To validate the provided theoretical results of the MSE as given by Theorem 1 and probability of support recovery as stated in Theorem 2, we consider the following example for the correlation matrix $\Sigmam$ \cite{shin2006capacity}: 
\begin{equation}\label{exp_corr}
\bSig(\rho) = \bigg[ \rho^{| i-j|^2} \bigg]_{i,j=1,2,\cdots,m}, \rho \in [0,1).
\end{equation}
For illustration, we focus only on the case where $\xv_0$ has entries that are sampled iid from a sparse-Bernoulli distribution  $p_{\xv_0} =(1-\kappa)\delta_0 + \kappa \delta_1$, where $\delta_{\cdot}$ is the Dirac delta function.
\figref{fig:mse_exp2} shows the MSE performance of the LASSO for different values of the regularizer $\lambda$. Monte Carlo Simulations are used to validate the theoretical prediction of Theorem 1. Comparing the simulation results to the asymptotic MSE prediction of Theorem 1 shows the close match between the two. We used $\delta =0.7, n=400, \rho =0.7, \sigma^2 = 0.01$, and $\kappa =0.1$, and the data are averaged over $500$ realizations of the channel matrix and the noise vector.

\begin{figure}
\begin{center}
\includegraphics[width =10.2cm]{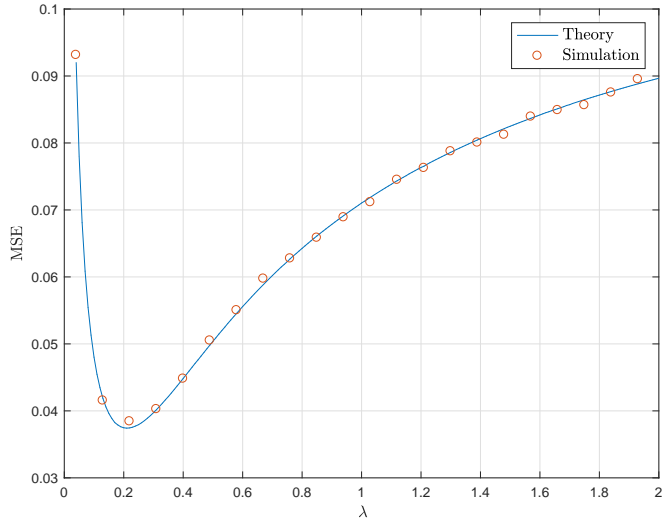}
\end{center}
\caption{\scriptsize {MSE performance of the LASSO decoder for the exponential correlation model in \eqref{exp_corr}, $\delta =0.7, n=400, \rho =0.7, \sigma^2 = 0.01$, and $\kappa =0.1$. For each $\lambda$ value, the data are averaged over $500$ independent realizations of the channel matrix, the signal vector and the noise vector.}}%
\label{fig:mse_exp2}
\end{figure}

In Fig 2 and Fig 3, we proivde the comparision between simulation and theory for the probability of successful on-support and off-support recovery respectively. We used the same values as in Fig 1. Again these figures show the preciseness of our results.

Fig 4 validates the prediction of Propostion 1 for the EER. This figure show the close agreement between simulation and Proposition 1. From this figure we can see that there is an optimal value of the regularizer $\lambda$ for which the EER is minimized.

Finally, Fig 5 shows the cosine similarity metric. As before, this figure show the precise nature of our results. As discussed earlier, we seek estimates that maximizes this measure and form this figure we can see a clear maximum value of the measure for some value of $\lambda$ around $0.14$.

 \begin{figure}
 \centering
\includegraphics[width=9cm, height =8.8cm]{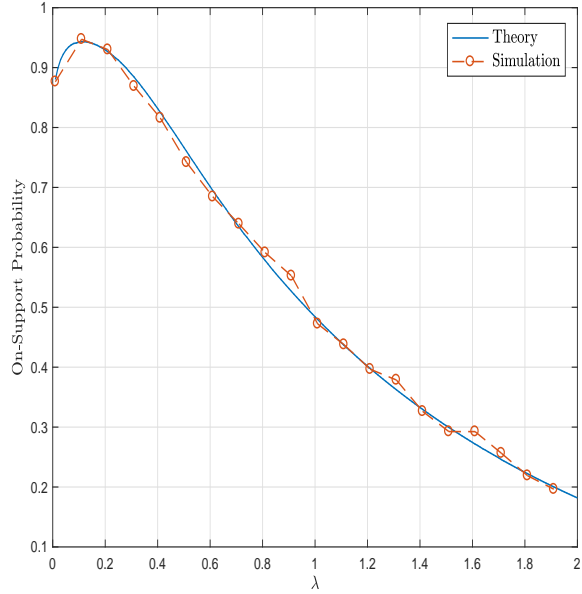}%
\caption{The probability of on-support recovery performance of the LASSO with $\xv_0$ being a sparse-Bernoulli vector. For simulations $\kappa =0.1, \delta = 0.7, n=400,  \xi = 0.001$, SNR = 10 dB.}%
\label{on_sup}
\end{figure}

 \begin{figure}
 \centering
\includegraphics[width=9cm, height =8.8cm]{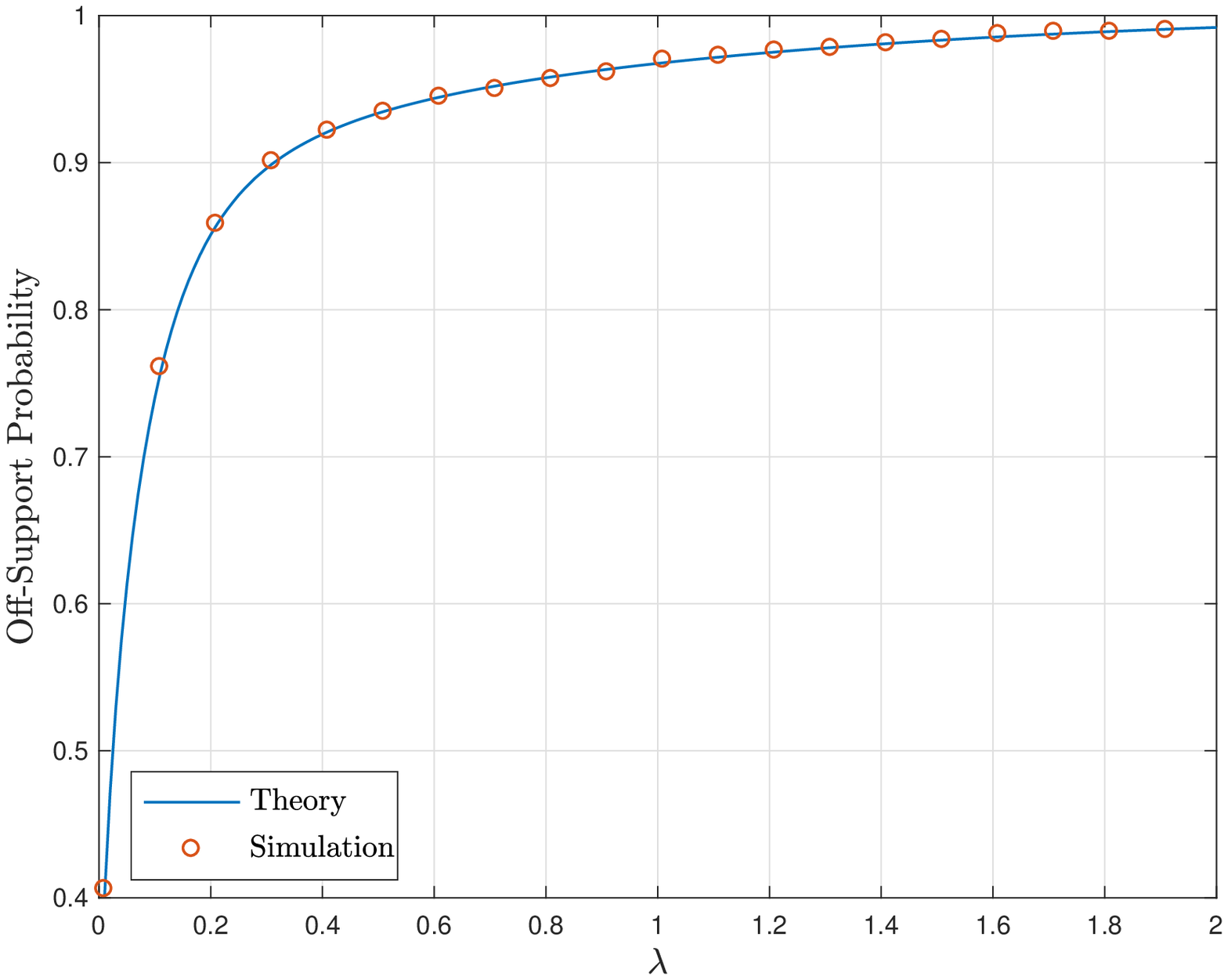}%
\caption{The probability of off-support recovery performance of the LASSO with $\xv_0$ being a sparse-Bernoulli vector. For simulations $\kappa =0.1, \delta = 0.7, n=400,  \xi = 0.001$, SNR = 10 dB.}%
\label{off_sup}
\end{figure}

\begin{figure}
\centering
\includegraphics[width=8.8cm, height =8cm]{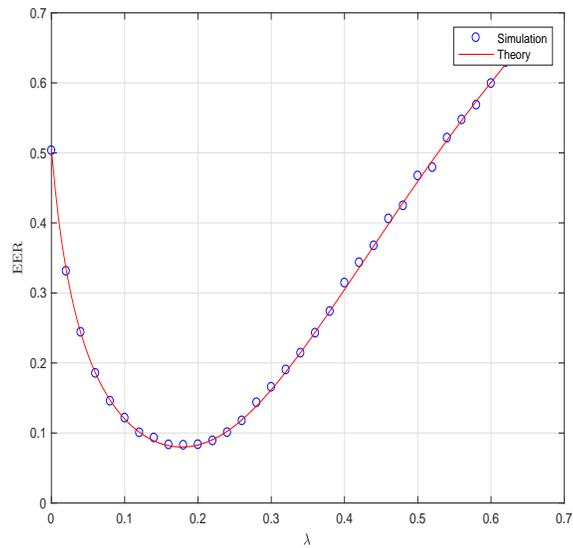}
\caption{\scriptsize {The EER performance of the LASSO. For simulations $\kappa =0.1,\delta = 0.7,n=400, \xi = 0.001$, SNR = 10 dB, and the data are averaged over 500 independent realizations of problem.}}%
\label{EER_Fig}
\end{figure}

 \begin{figure}
 \centering
\includegraphics[width=10.2cm, height =8cm]{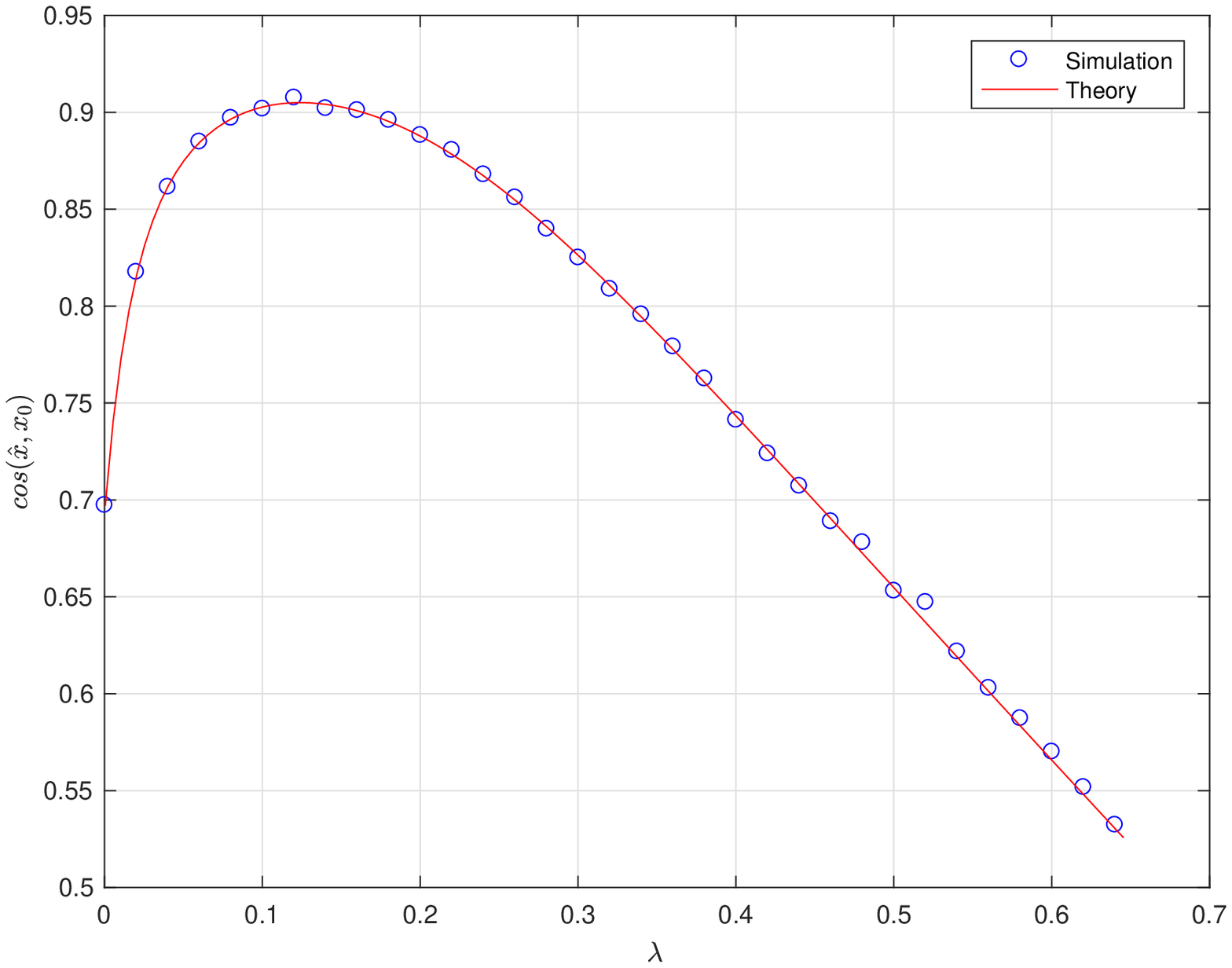}%
\caption{The cosine similarity performance of the LASSO with $\xv_0$ being a sparse-Bernoulli vector. For simulations $\kappa =0.1, \delta = 0.7,n=400$, SNR = 10 dB.}%
\label{cos_sim}
\end{figure}
\section{Approach and Proof Overview}\label{sec:proof}
In this section, we provide a proof outline of Theorems 1 and 2. The proof idea is mainly based on the framework of the CGMT which is summarized next.
\subsection{Convex Gaussian Min-max Theorem (CGMT)}
The key ingredient of the analysis is the CGMT. Here, we recall the statement of the theorem, and we refer the reader to \cite{thrampoulidis2018symbol, thrampoulidis2018precise} for the complete technical details.
Consider the following two min-max problems, which we refer to, respectively, as the Primary Optimization (PO) and Auxiliary Optimization (AO):
\begin{subequations}
\begin{align}
\label{P,AO}
&\Phi(\Cm) := \underset{\wv \in \mathcal{S}_{w}}{\operatorname{\min}}  \ \underset{\uv \in \mathcal{S}_{u}}{\operatorname{\max}} \ \uv^{T} \Cm \wv + \psi( \wv, \uv), \\
&\phi(\gv, \hv) := \underset{\wv \in \mathcal{S}_{w}}{\operatorname{\min}}  \ \underset{\uv \in \mathcal{S}_{u}}{\operatorname{\max}} \ \| \wv \| \gv^{T} \uv - \| \uv \| \hv^{T} \wv + \psi( \wv, \uv),\label{AA2} 
\end{align}
\end{subequations}
where $\Cm \in \mathbb{R}^{m \times n}, \gv \in \mathbb{R}^{m}, \hv \in \mathbb{R}^n, \mathcal{S}_w \subset \mathbb{R}^n, \mathcal{S}_u \subset \mathbb{R}^m$ and $\psi : \mathbb{R}^n \times \mathbb{R}^m \mapsto \mathbb{R}$. Denote by $\wv_{\Phi} := \wv_{\Phi}(\Cm) $ and $\wv_{\phi} := \wv_{\phi}( \gv, \hv)$ any optimal minimizers of (\ref{P,AO}) and (\ref{AA2}), respectively. Further let $\mathcal{S}_w, \mathcal{S}_u$ be convex and compact sets, $\psi(\cdot,\cdot)$ is convex-concave continuous on $\mathcal{S}_w \times \mathcal{S}_u$ and, $\Cm, \gv$ and $\hv $ all have iid standard normal entries.

Let $\mathcal{S}$ be any arbitrary open subset of $\mathcal{S}_w $, and $\mathcal{S}^c = \mathcal{S}_w \setminus\mathcal{S}$. Denote $\phi_{\mathcal{S}^c}(\gv,\hv)$ the optimal cost of the optimization in (\ref{AA2}), when the minimization over $\wv$ is constrained over $\wv \in \mathcal{S}^c$. Suppose that there exist constants $\bar{\phi}$ and $\eta >0$ such that in the limit as $n \rightarrow + \infty$, it holds with probability approaching one: (i) $\phi(\gv,\hv) \leq \bar{\phi} +\eta$, and, (ii) $\phi_{\mathcal{S}^c}(\gv,\hv) \geq \bar{\phi} + 2\eta$.
Then, $\lim_{n \rightarrow \infty} \mathbb{P}[\wv_{\phi} \in \mathcal{S}] = 1$, and $\lim_{n \rightarrow \infty} \mathbb{P}[\wv_{\Phi} \in \mathcal{S}] = 1.$ \\
\subsection{Identifying the PO and the AO}
For notational convenience, we consider the error vector $\wv := \xv- \xv_0 $,
then the problem in \eqref{eq:corr-LASSO} (after proper normalization by $n$) can be reformulated as
\begin{equation}\label{Lasso_w}
\widehat{\wv} := \argmin_{\wv} \frac{1}{n}\| \Am \wv -\zv \|_2^2 + \frac{\lambda}{n} \| \wv + \xv_0 \|_1.
\end{equation}
Using the invariance of the Gaussian distribution under orthogonal transformations, we have
\begin{equation}\label{Lasso2}
\widehat{\wv} = \argmin_{\wv} \frac{1}{n} \| \Gammam^{\frac{1}{2}} \Gm \wv -\zv \|_2^2 + \frac{\lambda }{n}\| \wv +\xv_0 \|_1,
\end{equation}
where $\Gm$ has iid Gaussian entries $\mathcal{N}(0,\frac{1}{n})$.
The loss function can be expressed in its dual form through the Fenchel conjugate as $$ \|   \Gammam^{\frac{1}{2}} \Gm \wv -\zv \|_2^2 = \max_{\uv} \sqrt{n} \uv^T (\Gammam^{\frac{1}{2}} \Gm \wv -\zv ) -\frac{n \| \uv \|_2^2}{4}.$$ 
Then, the PO can be written as
\begin{align}\label{}
\bPhi^{(n)}=&\frac{1}{n}\min_{ \wv} \max_{\uv}  \sqrt{n} \uv^T \Gammam^{\frac{1}{2}} \Gm \wv \nonumber \\
&  - \sqrt{n} \uv^T \zv -\frac{n \| \uv \|_2^2}{4} + \lambda \| \wv +\xv_0 \|_1.
\end{align}
Redefining $\uv$ as $\uv = \Gammam^{\frac{1}{2}} \uv$ yields
\begin{align}\label{PO_L}
\bPhi^{(n)}=\frac{1}{n} \min_{\wv} \max_{\uv} & \sqrt{n} \uv^T  \Gm \wv  - \sqrt{n} \uv^T \Gammam^{-\frac{1}{2}}\zv \nonumber \\
&-\frac{n }{4} \uv^T \Gammam^{-1} \uv + \lambda \| \wv +\xv_0 \|_1.
\end{align}
The above optimization is in a PO form, and its corresponding AO is
\begin{align}\label{AO_L}
\phi^{(n)}=\frac{1}{n} \min_{\wv} \max_{\uv} & \| \wv \|_2 \gv^T \uv- \| \uv \|_2 \hv^T \wv - \sqrt{n} \uv^T \Gammam^{-\frac{1}{2}}\zv \nonumber \\
  & -\frac{n }{4} \uv^T \Gammam^{-1} \uv + \lambda \| \wv +\xv_0 \|_1.
\end{align}
Recalling that $$ \| \av \|_1= \max_{\| \vv \|_{\infty} \leq 1} \av^T \vv,$$ for any vector $\av$, we have
\begin{align}\label{AO_LL}
\phi^{(n)}=\frac{1}{n}& \min_{\wv} \max_{\uv} \max_{\| \vv \|_{\infty} \leq 1} \| \wv \|_2 \gv^T \uv- \| \uv \|_2 \hv^T \wv \nonumber \\
  &- \sqrt{n} \uv^T \Gammam^{-\frac{1}{2}}\zv  -\frac{n }{4} \uv^T \Gammam^{-1} \uv + \lambda (\wv +\xv_0 )^T \vv.
\end{align}
Fixing the normalized norm of $\wv$ to $\sqrt{\alpha}= \frac{\| \wv \|_2}{\sqrt{n}}$, the AO can be expressed as
\begin{align}\label{AO_11}
 \phi^{(n)}=& \min_{\alpha\geq 0} \max_{\substack{\uv \\ \| \vv\|_\infty \leq 1}} \sqrt{\frac{\alpha}{n}} \gv^T \uv - \frac{1 }{\sqrt{n}} \uv^T \Gammam^{-\frac{1}{2}} \zv  -\frac{1 }{4} \uv^T \Gammam^{-1} \uv \nonumber \\
  &+\frac{\lambda}{n} \xv_0^T \vv +  \min_{{\| \bar{\wv} \|_2 =1}} \sqrt{\frac{\alpha}{n}} (\lambda \vv - \| \uv\|_2 \hv)^T \bar{\wv}.
\end{align}
The last minimization is easy to perform as $$\min_{{\| \bar{\wv} \|_2 =1}} \sqrt{\frac{\alpha}{n}}\big(\lambda \vv - \| \uv\|_2 \hv \big)^T \bar{\wv} = -\sqrt{\frac{\alpha}{n}} \big\| \lambda \vv - \| \uv \|_2 \hv \big\|_2,$$
with the optimal solution
\begin{equation}\label{w_*}
\bar{\wv}_* = - \frac{\lambda \vv - \| \uv \|_2 \hv}{\| \lambda \vv - \| \uv \|_2 \hv \|_2}.
\end{equation}
Then, we have the following 
\begin{align}\label{AO_12}
& \phi^{(n)}=\min_{\alpha\geq 0} \max_{\substack{\uv \\ \| \vv\|_\infty \leq 1}}  \sqrt{\frac{\alpha}{n}} \gv^T \uv - \frac{1 }{\sqrt{n}} \uv^T \Gammam^{-\frac{1}{2}} \zv  -\frac{1 }{4} \uv^T \Gammam^{-1} \uv \nonumber \\
  &+\frac{\lambda}{n} \xv_0^T \vv - \sqrt{\frac{\alpha}{n}}\big\| \lambda \vv - \| \uv \|_2 \hv \big\|_2.
\end{align}
Let $\tilde{\gv}=\sqrt{\alpha} \gv-  \Gammam^{-\frac{1}{2}} \zv$, and fixing the norm of $\uv$ to $\beta = \| \uv \|_2$, then
\begin{align}\label{AO_13}
& \phi^{(n)}=\min_{\alpha\geq 0} \max_{\substack{\beta\geq0 \\ \|\bar{\uv} \|_2 =1 \\ \| \vv\|_\infty \leq 1}}  \frac{1}{\sqrt{n}} \beta \tilde{\gv}^T \bar{\uv}  -\frac{\beta^2}{4} \bar{\uv}^T \Gammam^{-1} \bar{\uv} \nonumber \\
  &+\frac{\lambda}{n} \xv_0^T \vv - \sqrt{\frac{\alpha}{n}}\big\| \lambda \vv - \beta \hv \big\|_2.
\end{align}
Now, the optimization over $\bar{\uv}$ becomes separable, hence we need to solve the following non-convex problem
\begin{align}\label{u_1}
\max_{ \|\bar{\uv} \|_2 =1} & \frac{1}{\sqrt{n}} \beta \tilde{\gv}^T \bar{\uv}  -\frac{\beta^2}{4} \bar{\uv}^T \Gammam^{-1} \bar{\uv},
\end{align}
which is a standard optimization that has been extensively studied \cite{gander1989constrained,tao1998dc}. Its solution is $\bar{\uv}_* = \frac{2 }{\beta \sqrt{n}} \left( \Gammam^{-1} - \mu(\alpha,\beta) \Id \right)^{-1} \tilde{\gv}$, with $\mu(\alpha,\beta) $ satisfying
\begin{equation}\label{mu_eq}
\frac{1}{n}\tilde{\gv}^T \bigg( \Gammam^{-1} - \mu(\alpha,\beta) \Id \bigg)^{-2} \tilde{\gv} - \frac{\beta^2}{4} = 0.
\end{equation}
Subistituting $\bar{\uv}_* $ into \eqref{u_1} gives: $$ -\frac{\beta^2}{4} \mu(\alpha,\beta) + \frac{1}{n}\tilde{\gv}^T \bigg( \Gammam^{-1} - \mu(\alpha,\beta) \Id \bigg)^{-1} \tilde{\gv}.$$
Therefore, the AO becomes
{\small{
\begin{align}\label{AO_14}
 \phi^{(n)}=\min_{\alpha\geq 0} \max_{\substack{\beta\geq0  \\ \| \vv\|_\infty \leq 1}} &  -\frac{\beta^2}{4} \mu(\alpha,\beta) + \frac{1}{n}\tilde{\gv}^T \bigg( \Gammam^{-1} - \mu(\alpha,\beta) \Id \bigg)^{-1} \tilde{\gv} \nonumber \\
  &+\frac{\lambda}{n} \xv_0^T \vv -  \big\| \sqrt{\frac{\alpha}{n}} (\lambda \vv - \beta \hv) \big\|_2,
\end{align}}}
where $\mu(\alpha,\beta) $ satisfies \eqref{mu_eq}. We proceed by expressing the $\ell_2$-norm in \eqref{AO_14} using the following varational form
$$
\|  \sv \|_2 = \inf_{\chi >0} \frac{\chi}{2} + \frac{\|  \sv \|_2^2}{2 \chi},
$$
for any vector $\sv$.
{\small{
\begin{align}\label{AO_15}
 \phi^{(n)}=& \min_{\alpha\geq 0} \max_{\substack{\beta\geq0  \\ \| \vv\|_\infty \leq 1}} \sup_{\chi>0}  -\frac{\beta^2}{4} \mu(\alpha,\beta) + \frac{1}{n}\tilde{\gv}^T \bigg( \Gammam^{-1} - \mu(\alpha,\beta) \Id \bigg)^{-1} \tilde{\gv} \nonumber \\
  &+\frac{\lambda}{n} \xv_0^T \vv - \frac{ \chi}{2} - \frac{\alpha}{2 \chi} \big\| \frac{1}{\sqrt{n}}\big(\lambda \vv - \beta \hv \big) \big\|_2^2.
\end{align}}}
Next, we perform the optimization over $\vv$, since it is separable now. So, we need to solve:
\begin{align}\label{v_eq}
\max_{\substack{ \| \vv\|_\infty \leq 1}}  \frac{\lambda}{n} \xv_0^T \vv -\frac{\alpha}{2 \chi} \big\| \frac{1}{\sqrt{n}}\big(\lambda \vv - \beta \hv \big) \big\|_2^2.
\end{align}
This can be rewritten as
\begin{align}\label{v_eq2}
\frac{\chi}{\alpha n} \sum_{i =1}^{n}\max_{-1\leq v_i \leq 1}  \bigg(\frac{\lambda \alpha}{\chi} x_{0,i} + \frac{\alpha^2 \beta \lambda}{\chi^2} h_i\bigg)v_i - \frac{\alpha^2 \lambda^2}{2 \chi^2} v_i^2 - \frac{\alpha^2 \beta^2}{2 \chi^2} h_i^2.
\end{align}
Let $a_i = \frac{\lambda \alpha}{\chi} x_{0,i} +  \frac{\alpha^2 \beta \lambda}{\chi^2} h_i.$ Then, the optimal solution is given by:
\begin{equation}\label{v_opt}
v_i^*=\begin{cases} -1, \ \ \ \ \  \   \ \ \ {\rm if}\ \  \left(\frac{\chi}{\alpha \lambda}\right)^2 a_i<-1\\
\left(\frac{\chi}{\alpha \lambda}\right)^2 a_i,\ \ \  {\rm if}\ \ -1\leq \left(\frac{\chi}{\alpha \lambda}\right)^2 a_i \leq1\\
1,   \ \ \  \ \ \ \ \ \ \ \ \ \ {\rm if}\ \ \left(\frac{\chi}{\alpha \lambda}\right)^2 a_i>1.\end{cases}
\end{equation}
Subistituting $v_i^*$ into \eqref{v_eq2} and after some algebraic manipulations we get
\begin{align*}
\frac{\chi}{\alpha n} \bigg[ \sum_{i =1}^{n} e \left(x_{0,i} + \frac{\alpha \beta}{\chi} h_i ; \frac{\lambda \alpha}{\chi} \right)- \frac{\alpha^2 \beta^2}{2 \chi^2} h_i^2 \bigg],
\end{align*}
where $e(\cdot;\cdot)$ is as defined in \eqref{soft_cost}. Now, the AO becomes
{{
\begin{align}\label{AO_16}
  \tilde\phi^{(n)}=\min_{\alpha\geq 0} \max_{\substack{\beta\geq0 }} \sup_{\chi>0}&  -\frac{\beta^2}{4} \mu(\alpha,\beta) + \frac{1}{n}\tilde{\gv}^T \bigg( \Gammam^{-1} - \mu(\alpha,\beta) \Id \bigg)^{-1} \tilde{\gv} \nonumber \\
  & - \frac{\chi}{2} +\frac{1}{n} \bigg[ \sum_{i =1}^{n} \frac{\chi}{\alpha } e \left(x_{0,i} + \frac{\alpha \beta}{\chi} h_i ; \frac{\lambda \alpha}{\chi} \right)- \frac{\alpha \beta^2}{2 \chi} h_i^2 \bigg].
\end{align}}}
\subsection{Probabilistic Asymptotic Analysis of the AO}
Note that $\tilde \bg \sim\mathcal{N}(\boldsymbol{0},\bC_{\tilde\bg})$, where $\bC_{\tilde\bg}$ is given by
$$
\bC_{\tilde\bg}=\alpha \bI_m+\sigma^2\Gammam^{-1}.
$$
Thus, applying the trace lemma \cite{Couillet2011}, we have
{{
\begin{align*}
\frac{1}{n}\tilde{\gv}^T \bigg( \Gammam^{-1} - \mu(\alpha,\beta) \Id \bigg)^{-1} \tilde{\gv}  - \frac{1}{n}\tr\left(\bC_{\tilde\bg} \bigg( \Gammam^{-1} - \mu(\alpha,\beta) \Id \bigg)^{-1}\right)\pto 0.
\end{align*}}}
Also, using the WLLN, $\frac{1}{n} \sum_i h_i^2 \pto 1$, and for all $\alpha \geq 0, \beta>0$ and $\chi >0$, we have
$\frac{1}{n} \sum_i e \left(x_{0,i} + \frac{\alpha \beta}{\chi} h_i ; \frac{\lambda \alpha}{\chi} \right) \pto \mathbb{E}_{\underset{Z \sim \mathcal{N}(0,1)}{X_0 \sim p_{X_0}} } \biggr[e \biggr(X_0 + \frac{\alpha \beta }{ \chi} Z ; \frac{ \lambda \alpha}{\chi}\biggl) \biggr]$.
Therefore, again using Lemma 10 of \cite{thrampoulidis2018precise}, $\tilde\phi^{(n)}- \overline\phi^{(n)} \pto 0,$ where
{{
\begin{align}\label{AO_17}
\overline \phi^{(n)}=&\min_{\alpha\geq 0} \max_{\beta\geq0}\sup_{\chi >0} -\frac{\beta^2}{4} \mu(\alpha,\beta) -\frac{\chi}{2} -\frac{\alpha \beta^2}{2 \chi} \nonumber \\
+&\frac{1}{n}\sum_{j=1}^{m} \frac{\gamma_j \alpha + \sigma^2}{1 - \gamma_j  \mu(\alpha,\beta) } + \frac{\chi}{\alpha} \mathbb{E}_{\underset{Z \sim \mathcal{N}(0,1)}{X_0 \sim p_{X_0}} } \biggr[e \biggr(X_0 + \frac{\alpha \beta }{ \chi} Z ; \frac{ \lambda \alpha}{\chi}\biggl)    \biggl] ,
\end{align}
}}
where $\mu(\alpha,\beta)$ satisfies (from \eqref{mu_eq} and using the trace Lemma)
\begin{align}\label{mu_eq2}
\frac{1}{n} \sum_{j=1}^{m} \frac{\alpha + \frac{\sigma^2}{\gamma_j}}{\left(\frac{1}{\gamma_j}- \mu(\alpha,\beta)\right)^2} -\frac{\beta^2}{4} =0.
\end{align}
\subsection{Applying the CGMT}
Now we will evaluate the performance of the LASSO using the different metrics introduced earlier. We begin with the MSE analysis. 
Let $\widetilde \wv$ be the optimal solution to the AO defined as the solution to \eqref{AO_11}. 
Let $\alpha_\star$ be the optimal solution to \eqref{AO_17}. For any $\epsilon>0$, define the set:
$$
\mathcal{S}_{\epsilon} = \bigg\{ \rv: \bigg| \frac{1}{n} \| \rv \|_2^2 - \alpha_\star \biggr| < \epsilon \bigg\}.
$$
Define $\hat{\alpha}_n$ as the minimizer of \eqref{AO_11}. Then, by definition, $\hat{\alpha}_n = \frac{\| \widetilde \wv \|_2^2}{n}$. In the previous section, we showed that $\phi^{(n)} - \overline \phi^{(n)} \pto 0$. Hence, we can show that $\hat{\alpha}_n - \alpha_\star \pto 0$, which implies
$$
\bigg| \frac{1}{n} \| \widetilde \wv \|_2^2 - \alpha_\star \bigg| \pto 0.
$$
Therefore, $\widetilde\bw \in \mathcal{S}_{\epsilon} $ with probability approaching 1. Then, applying the CGMT yields that $\widehat{\wv} \in \mathcal{S}_{\epsilon}$ with probability approaching 1 as well. This completes the proof of Theorem \ref{LASSO_mse}.

We proceed now to the proof of the probabilities of support recovery. First, for the on-support recovery probability, change the set to the following:
{{
$$
\mathcal{S}_{\epsilon} = \bigg\{ \rv: \bigg| \frac{1}{k} \sum_{i \in S(\xv_0)} \mathbbm{1}_{\{| r_{i}| \geq \xi \}} - \mathbb{P} \biggl[\biggl | \eta \biggr(X_0 + \frac{\alpha_\star \beta_\star}{\chi_\star} Z ; \frac{ \lambda \alpha_\star }{\chi_\star}\biggl)  \biggr |   \geq \xi \biggr]\biggr| < \epsilon \bigg\},
$$}}
for any $\xi>0$.
Note that it can be shown, based on \eqref{w_*}, that for all $i =1,2,\cdots,n$: 
\begin{align}
\widetilde{w}_i = -\frac{\tilde{\alpha} (\lambda v_i^* -  \tilde{\beta} h_i)}{\sqrt{\frac{\tilde{\alpha}}{n}} \| \lambda v_i^* - \tilde{\beta} h_i \|_2},
\end{align}
where $\tilde{\alpha}, \tilde{\beta}$ are the solutions of \eqref{AO_16}. Note that $\sqrt{\frac{\tilde{\alpha}}{n}} \| \lambda \vv_i^* - \tilde{\beta} h_i \|_2 = \tilde{\chi}$ which is the solution of \eqref{AO_16} as well. Then
\begin{align*}
\widetilde{w}_i = -\frac{\tilde{\alpha} (\lambda v_i^* -  \tilde{\beta} h_i)}{\tilde{\chi}}.
\end{align*}
Recall that $\widetilde \wv = \widetilde \xv - \xv_0$, where $\widetilde \wv$ is the AO solution. Hence
$$
\widetilde{x}_i = \widetilde w_i + x_{0,i}. 
$$
Subistituting the values of $\widetilde w_i $ and $v_i^*$ and after some algebraic manipulations, it can be shown that $\widetilde{x}_i = \eta(x_{0,i} + \frac{\tilde\alpha \tilde\beta}{\tilde\chi} h_i ; \frac{\lambda \tilde\alpha}{\tilde\chi})$.
Since $\tilde \phi^{(n)} - \overline \phi^{(n)} \pto 0$, it can be shown that $\tilde\alpha - \alpha_\star \pto 0, \tilde\beta - \beta_\star \pto 0,$ and $\tilde\chi - \chi_\star \pto 0$. Then, after some simple calculations, it holds
$$
\bigg| \frac{1}{k} \sum_{i \in S(\xv_0)} \mathbbm{1}_{\{| \tilde x_i| \geq \xi \}} - \mathbb{P} \biggl[\biggl | \eta \biggr(X_0 + \frac{\alpha_\star \beta_\star}{\chi_\star} Z ; \frac{ \lambda \alpha_\star }{\chi_\star}\biggl)  \biggr |   \geq \xi \biggr]\biggr|  \pto 0.
$$
This proves that $\widetilde\bx \in \mathcal{S}_{\epsilon} $ with probability approaching 1.
Note that the indicator function $\mathbbm{1}_{\{|\tilde x_i|\geq \xi\}} $ is not Lipschitz, so we cannot directly apply the CGMT. However, as discussed in \cite[Lemma A.4]{thrampoulidis2018symbol} and \cite{alrashdi2020optimum}, this function can be appropriately approximated with Lipschitz functions. Therefore, we can conclude by applying the CGMT that $\widehat{\bx} \in \mathcal{S}_{\epsilon}$ with probability approaching 1, which proves the first result of Theorem \ref{lasso_on/off}. The off-support recovery probability can be derived in a similar manner and details are thus omitted.
\section{Conclusion}\label{sec:conclusion}
In this paper, we derived precise asymptotic error performance analysis of LASSO under the assumption that the design matrix has correlated entries. In particular, we derived precise expressions of the MSE, probability of support recovery, EER, and cosine similarity. Numerical simulations show the close agreement to the theory even for low dimensions of the problem. Possible future extensions inculde the double-sided correlation model, imperfect channel models and analyzing the box varient of the LASSO.
\begin{appendices}
\section*{Appendix}
The expectation in \eqref{AO_17} can be evaluated in closed form for any distribution. For example, take the case of a sparse-Bernoulli vector $\xv_0$, i.e., the entries of $\xv_0$ are sampled iid from a distribution $p_{\xv_0} =(1-\kappa)\delta_0 + \kappa \delta_1$, then 
{\small{
\begin{align*}
&\frac{\chi}{\alpha}\mathbb{E}_{\underset{Z \sim \mathcal{N}(0,1)}{X_0 \sim p_{X_0}} } \biggr[e \biggr(X_0 + \frac{\alpha \beta }{ \chi} Z ; \frac{ \lambda \alpha}{\chi}\biggl)    \biggl] = \nonumber \\
&\frac{\alpha (1- \kappa)}{\chi}\left(\frac{\beta^2}{2}+ \beta \lambda  \varphi \bigg(\frac{\lambda}{ \beta} \bigg) - (\lambda^2+\beta^2) Q \left(\frac{\lambda}{\beta} \right) \right)  \nonumber \\
&+ \kappa \left( \lambda- \frac{\alpha \lambda^2}{2 \chi} \right) Q\left(\frac{\lambda}{\beta} - \frac{\chi}{\alpha \beta} \right) - \kappa \left( \lambda + \frac{\alpha \lambda^2}{2 \chi}\right) Q \left(\frac{\lambda}{\beta} + \frac{\chi}{\alpha \beta} \right) \nonumber \\
& +\frac{\alpha \beta \lambda \kappa}{\chi } \left( \varphi  \left(\frac{\lambda}{\beta} +\frac{\chi}{\alpha \beta} \right) +  \varphi \left(\frac{\lambda}{\beta} -\frac{\chi}{\alpha \beta}  \right) \right) \nonumber \\
&- \frac{\kappa \beta}{2  \chi } \varphi  \left( \frac{\alpha \lambda + \chi}{\alpha \beta}  \right) \left( \alpha \lambda - \chi + (\alpha \lambda + \chi ) \exp \left( \frac{2 \lambda \chi}{\alpha \beta^2} \right) \right) \nonumber \\
& + \frac{\kappa}{4} \left(\frac{\alpha \beta^2}{\chi} +\frac{\chi}{ \alpha } \right)\left( {\rm erf} \left(\frac{\alpha \lambda +\chi}{\sqrt{2} \alpha \beta} \right) + {\rm erf} \left(\frac{\alpha \lambda -\chi}{\sqrt{2} \alpha \beta} \right) \right),
\end{align*}
}}
where ${\rm{erf}}(x)$ is the error function defined as ${\rm {erf}}(x) = \frac{2}{\sqrt{\pi}}\int_{0}^{x} e^{-t^2} {\rm d}  t$.

Also, for the sparse-Bernoulli distribution, we have
$$
\underset{n\to\infty}{{\rm{plim}}}\ \Phi_{\xi,\rm{on}}(\widehat{\xv})  = Q \left(\frac{\lambda}{\beta_\star} +\frac{\chi_\star(\xi +1)}{\alpha_\star \beta_\star} \right) + Q \left( \frac{\lambda}{\beta_\star} +\frac{\chi_\star(\xi -1)}{\alpha_\star \beta_\star}  \right),
$$
 
$$
\underset{n\to\infty}{{\rm{plim}}}\ \Phi_{\xi,\rm{off}}(\widehat{\xv})  = 1- 2 Q \left(\frac{\lambda}{\beta_\star} +\frac{\chi_\star \xi }{\alpha_\star \beta_\star}  \right).
$$
and
\begin{align*}
\underset{{n \to \infty}}{\rm{plim}} {\rm{EER}}_{\xi} =  Q \left(\frac{\chi_\star (1-\xi)}{\alpha_\star \beta_\star} -\frac{\lambda}{\beta_\star}\right)  
- Q \left(\frac{\chi_\star (1+\xi)}{\alpha_\star \beta_\star} 
+\frac{\lambda}{\beta_\star} \right) + 2 Q \left( \frac{\chi_\star \xi}{\alpha_\star \beta_\star} +\frac{\lambda}{\beta_\star} \right). 
\end{align*}

Finally, for the cosine similarity, we have for the sparse-Bernoulli distribution:
$$
\underset{n\to\infty}{{\rm{plim}}}\ \cos(\angle (\widehat{\xv},\xv_0)) = \frac{  I_0 }{\sqrt{\kappa(I_1 +I_2)}},
$$
where 
$$
I_0 = \frac{\kappa}{\chi_\star} \left[ \alpha_\star \beta_\star \left( \varphi \left( \frac{\chi_\star}{\alpha_\star \beta_\star} -\frac{\lambda}{\beta_\star}\right) - \varphi \left( \frac{\chi_\star}{\alpha_\star \beta_\star} +\frac{\lambda}{\beta_\star}\right)\right) + (\chi_\star - \lambda \alpha_\star ) Q \left(\frac{\lambda}{\beta_\star} -\frac{\chi_\star}{\alpha_\star \beta_\star} \right) +(\chi_\star + \lambda \alpha_\star) Q \left(\frac{\lambda}{\beta_\star} +\frac{\chi_\star}{\alpha_\star \beta_\star} \right) \right],
$$
{\tiny{
$$
I_1 = \frac{\kappa}{\chi_\star^2} \left[ \bigg(\alpha_\star^2\beta_\star^2 +(\lambda \alpha_\star-\chi_\star)^2\bigg) Q \left(\frac{\lambda}{\beta_\star} - \frac{\chi_\star }{\alpha_\star \beta_\star} \right) + \bigg(\alpha_\star^2\beta_\star^2 +(\lambda \alpha_\star + \chi_\star)^2\bigg) Q \left(\frac{\lambda}{\beta_\star} + \frac{\chi_\star }{\alpha_\star \beta_\star} \right) -  \alpha_\star \beta_\star \left( (\lambda \alpha_\star  - \chi ) \exp \left( \frac{2 \lambda \chi_\star}{\alpha_\star \beta_\star^2} \right) + \lambda \alpha_\star + \chi_\star \right) \varphi \left(\frac{\lambda}{\beta_\star} + \frac{\chi_\star }{\alpha_\star \beta_\star} \right) \right],
$$}}
and 
$$
I_2 = \frac{2 (1-\kappa) \alpha_\star^2}{\chi_\star^2} \left[ (\lambda^2 +\beta_\star^2) Q \left( \frac{\lambda}{\beta_\star}\right) - \lambda \beta_\star \varphi \left( \frac{\lambda}{\beta_\star} \right)\right].
$$

These expressions were used in Section \ref{sec:proof} for the provided numerical results.
\end{appendices}
\bibliographystyle{IEEEbib}
\bibliography{References.bib}
\end{document}

%% file: macros.tex
\long\def\comment#1{}

\DeclareMathOperator*{\argmin}{arg\,min}

\newfont{\bbb}{msbm10 scaled 700}

\newfont{\bb}{msbm10 scaled 1100}

\newcommand{\av}{{\bf a}}

\newcommand{\gv}{{\bf g}}
\newcommand{\hv}{{\bf h}}

\newcommand{\rv}{{\bf r}}
\newcommand{\sv}{{\bf s}}

\newcommand{\uv}{{\bf u}}
\newcommand{\wv}{{\bf w}}
\newcommand{\vv}{{\bf v}}
\newcommand{\xv}{{\bf x}}
\newcommand{\yv}{{\bf y}}
\newcommand{\zv}{{\bf z}}

\newcommand{\Am}{{\bf A}}

\newcommand{\Cm}{{\bf C}}

\newcommand{\Gm}{{\bf G}}
\newcommand{\Hm}{{\bf H}}
\newcommand{\Id}{{\bf I}}

\newcommand{\Um}{{\bf U}}

\newcommand{\Gammam}{\hbox{\boldmath$\Gamma$}}

\newcommand{\Sigmam}{\hbox{\boldmath$\Sigma$}}

